\documentclass[12pt,oneside]{amsart}
\usepackage{amsthm,verbatim,mathtools,cite}
\usepackage{amssymb,latexsym,longtable}
\usepackage{amsmath}
\usepackage{graphicx}
\usepackage{color}
\usepackage[dvipsnames]{xcolor}
\usepackage{cancel}
\usepackage{soul}
\usepackage{enumerate}
\newtheorem{theorem}{Theorem}[section]
\newtheorem{definition}[theorem]{Definition}

\newtheorem{lemma}[theorem]{Lemma}
\newtheorem{corollary}[theorem]{Corollary}
\newtheorem{proposition}[theorem]{Proposition}
\newtheorem{ex}[theorem]{Example}

\newcommand{\mdan}[1]{{{\color{blue}#1}}}

\def\ariane#1 {\fbox {\footnote {\ }}\ \footnotetext { From Ariane: {\color{red}#1}}}
\def\daniele#1 {\fbox {\footnote {\ }}\ \footnotetext { From Daniele: {\color{blue}#1}}}
\def\lu#1 {\fbox {\footnote {\ }}\ \footnotetext { From Luciane: {\color{YellowOrange}#1}}}

\begin{document}
\title{Permutation polynomials over $\mathbb{F}_{q^2}$ from rational functions }
\author{Daniele Bartoli, Ariane M. Masuda, and Luciane Quoos}
\date{\today}
\address{Dipartimento di Matematica e Informatica, Universit\`a degli Studi di Perugia, Via Vanvitelli 1, Perugia, 06123   Italy}

\email {daniele.bartoli@unipg.it}

\address{Department of Mathematics, New York City College of Technology, CUNY, 300 Jay Street, Brooklyn, NY 11201 USA}

\email{amasuda@citytech.cuny.edu}

\address{Instituto de Matem\'atica, Universidade Federal do Rio de Janeiro,  Av. Athos da Silveira Ramos 149, Centro de Tecnologia - Bloco C, Ilha do Fund\~ao, Rio de Janeiro, RJ 21941-909  Brazil}

\email{luciane@im.ufrj.br}
\maketitle

\begin{abstract}
Let $\mu_{q+1}$ denote the set of $(q+1)$-th roots of unity in $\mathbb{F}_{q^2 }$. We construct permutation polynomials over $\mathbb{F}_{q^2}$ by using rational functions of any degree that induce bijections either on $\mu_{q+1}$  or between $\mu_{q+1}$ and $\mathbb{F}_q \cup \{\infty\}$. In particular, we generalize results from Zieve.
\end{abstract}

{\let\thefootnote\relax}

\section{Introduction}

Let $\mathbb{F}_q$ be the finite field of $q$ elements. A polynomial is said to be a {\it permutation polynomial} over $\mathbb{F}_q$ if the associated function induces a one-to-one map from $\mathbb{F}_q$ to itself. Over the past decades the research on permutation polynomials has been intensified due to their applications in many areas such as coding theory~\cite{L2007} and cryptography~\cite{LM1984,SH1998}. A problem of great interest is to find families of permutation polynomials over a finite field.  For an introduction on the basic properties, constructions and applications, we refer the reader to~\cite{LN1997,M1993, MP2013,Hou2015}.

For  a divisor $m \geq 2$ of $q-1$, let $\mu_{m}$ denote the set of $m$-th roots of unity in $\mathbb{F}_q$. Recently, many authors have used the following criterion, due to Tucker and Zieve~\cite{TZ2000}, to determine whether or not a polynomial permutes $\mathbb F_q$.

\begin{theorem}\label{ppcriterion}
Let $h\in\mathbb F_q[x]$ and integers $d,r>0$ such that $d\mid (q-1)$. Then $x^rh(x^{d})$ permutes $\mathbb F_q$ if and only if
$(r,d)=1$ and $x^rh(x)^{d}$ permutes $\mu_{(q-1)/d}$.
\end{theorem}

This method relies on polynomials permuting $\mu_m$ for certain values of $m$. In~\cite{Zieve2013} Zieve introduces a variant of this
approach based on functions that are induced by rational functions that permute $\mu_m$.  He completely classifies all degree-one rational functions over $\overline{\mathbb{F}}_{q}$ that are bijections $\mu_{q+1}\to \mu_{q+1}$ and $\mu_{q+1}\to\mathbb{F}_q \cup \{\infty \}$; here $\mu_{q+1}$ is the set of $(q+1)$-th roots of unity in $\mathbb{F}_{q^2}$. As a consequence, he obtains the following classes of permutation polynomials over $\mathbb{F}_{q^2}$.

\begin{theorem}\label{Zieve2013Theorem1.1}
Let  $n>0$ and $k\geq 0$ be integers, and let $\beta,\gamma\in\mathbb{F}_{q^2}$ satisfy $\beta^{q+1}=1$ and $\gamma^{q+1}\neq 1$. Then
$$x^{n+k(q+1)} ((\gamma x^{q-1}-\beta)^n-\gamma(x^{q-1}-\gamma^{q}\beta)^{n})$$
permutes $\mathbb{F}_{q^2}$ if and only if $(n+2k,q-1)=1$ and $(n,q+1)=1$.
\end{theorem}

\begin{theorem}\label{Zieve2013Theorem1.2}
Let $n>0$ and $k\geq 0$ be integers, and let $\beta, \delta\in\mathbb{F}_{q^2}$ satisfy $\beta^{q+1}=1$ and  $\delta \notin \mathbb{F}_q$. Then
$$x^{n+k(q+1)} ((\delta x^{q-1}-\beta\delta^q)^n-\delta(x^{q-1}-\beta)^{n})$$
permutes $\mathbb{F}_{q^2}$ if and only if $(n(n+2k),q-1)=1$.
\end{theorem}

In this work we present the notion of a pair of polynomials $(L,M)$ being \emph{$\beta$-associated} in  $\mathbb{F}_{q^2}[x]$ where $L$ and $M$ have the same degree. A similar concept is that of $L$ being  \emph{$(\beta,t)$-self-associated}. Our key observation is that the degree-one rational functions $L/M$ that are bijections $\mu_{q+1}\to \mu_{q+1}$ are such that $(L,M)$ is $\beta$-associated.  In addition, the degree-one rational functions $L/M$ that are bijections $\mu_{q+1}\to\mathbb{F}_q \cup \{\infty \}$ are such that $L$
and $M$ are $(\beta,t)$-self-associated. By investigating  $\beta$-associated polynomials and $(\beta,t)$-self-associated-polynomials of degree greater than one in $\mathbb F_{q^2}[x]$, we  are able  to extend Theorems~\ref{Zieve2013Theorem1.1} and~\ref{Zieve2013Theorem1.2}.
In addition, we construct several new families of permutation polynomials arising from degree-$n$ rational functions with $n>1$.

 We use a standard approach to verify that a certain rational function satisfies a permuting condition.  More specifically, in order to check that a rational function $F$   induces a bijection on $\mu_{q+1}$,
we consider  the plane algebraic curve of affine equation
$$
\mathcal C_F:\frac{F(x)-F(y)}{x-y}=0.
$$
Then $F$ permutes $\mu_{q+1}$ if and only if  $\mathcal C_F$ has no points $(a,b)\in \mu_{q+1}^2$ with $a\neq b$. We follow a similar
idea when we require a bijection between $\mu_{q+1}$ and $\mathbb{F}_q \cup \{\infty \}$. An effective method to check the existence of those  particular $\mathbb{F}_{q^2}$-rational points is to determine the absolutely irreducible components of $\mathcal{C}_F$; see \cite{BG2018,Bartoli2018}.



Our paper is organized in the following way.  In Section 2 we introduce the idea of  \emph{$\beta$-associated} and  \emph{($\beta$,t)-self-associated} for polynomials in $\mathbb F_{q^2}[x]$. We list some immediate properties that we use throughout the paper. In Sections~\ref{Sec3} and~\ref{Sec5} we present our generalizations of Theorems~\ref{Zieve2013Theorem1.1} and~\ref{Zieve2013Theorem1.2}, respectively; see Theorems~\ref{prop:good} and~\ref{mainTheorem2}. Our results are followed by several explicit constructions of  polynomials that can be used to form new classes of permutation polynomials over
$\mathbb F_{q^2}$.  In Section~\ref{Sec4} we provide a generalization of our own results for other extension fields, $\mathbb{F}_{q^k}$,
leading to new permutation polynomials over $\mathbb F_{q^3}$.

\section{$\beta$-associated polynomials}
In~\cite{Zieve2013} Zieve describes precisely the degree-one rational functions over $\overline{\mathbb F}_q$ that are bijections when restricted to $\mu_{q+1}\to \mu_{q+1}$ and $\mu_{q+1}
\to\mathbb{F}_q \cup \{\infty\}$.

\begin{proposition}\label{Zieve2013,Lemm A_2.1}
Let  $\ell\in\overline{\mathbb F}_q(x)$ be a degree-one rational function.  Then $\ell$ induces a bijection
on $\mu_{q+1}$ if and only if $\ell$ equals either
 $\beta/x$ with $\beta \in\mu_{q+1}$ or
$(x-\gamma^q\beta)/(\gamma x-\beta)$ with $\beta\in\mu_{q+1}$ and $\gamma\in\mathbb F_{q^2}\setminus \mu_{q+1}$.
\end{proposition}

\begin{proposition}\label{bijinfty}
Let $\ell \in \overline{\mathbb{F}}_q(x)$ be a degree-one rational function. Then $\ell$ induces a bijection from $\mu_{q+1}$ to $\mathbb{F}_q \cup \{\infty $\} if and only if $\ell(x)=(\rho x+\rho^q)/(\epsilon x+\epsilon^q)$ for some $\rho, \epsilon \in \mathbb{F}_{q^2}^{\;*}$ satisfying $\rho^{q-1}\neq\epsilon^{q-1}$.
\end{proposition}

We note that the rational functions $L/M$ that appear in Propositions~\ref{Zieve2013,Lemm A_2.1} and~\ref{bijinfty} satisfy the identity,
$$L^q=\beta x^{-1}M$$
for every  $x \in \mu_{q+1}$ and some $\beta \in \mathbb{F}_{q^2}$.
This observation is key to our work.  By considering polynomials $L$ and $M$ of any degree rather than one only, we
introduce the following.

\begin{definition}\label{Definition:associated}
Let $L, M \in \mathbb{F}_{q^2}[x]$  be distinct polynomials of the same degree, and let $\beta\in\mu_{q+1}$. When $L^{q}=\beta x^{-\deg(L)}M$ for every $x \in \mu_{q+1}$, we say that $L$ and $M$ are \emph{$\beta$-associated}  and write $L\sim_{\beta} M$.
When $L^{q}=\beta x^{-t}L$ for every $x\in \mu_{q+1}$ and some $t\in\mathbb N$, we say that $L$ is \emph{$(\beta,t)$-self-associated} and write $L\sim_{\beta,t} L$.
\end{definition}

One can see that associated polynomials satisfy some straightforward properties.

\begin{lemma}\label{assoc_prop}
\begin{enumerate}[{\normalfont(i)}]
\item If $L\sim_{\beta} M$ then $M\sim_{\beta} L$.
\item If $L\sim_{\beta} M$ and $M\sim_{\gamma}N$ then $LM \sim_{\beta\gamma}MN$.
\item If $L\sim_{\beta} M$ then $L^n \sim_{\beta^n} M^n$ for any positive integer $n$.
\item If $L$ is a self-reciprocal polynomial over $\mathbb F_q$ then $L\sim_{1,\deg L}L$.
\end{enumerate}
\end{lemma}

In the next sections we will exhibit several examples of associated polynomials of degree greater than one. Before we present them, we will set one more notation.  Suppose that $L, M, N \in \mathbb{F}_{q^2}[x]$ with $L$ and $M$ of the same degree. We define the polynomial
$$N \bullet  L / M:=M^{\deg(N)}N\left(L/M\right)
$$
whose degree is $\deg(M)\cdot\deg(N)$.

\begin{proposition}\label{prop:composition}
If $L\sim_{\beta} M$ and $\widetilde L\sim_{\gamma} \widetilde M$ then
$L \bullet \widetilde L /\widetilde M \sim_{\beta \gamma^{\deg(L)}}M\bullet \widetilde L /\widetilde M.$
\end{proposition}

\begin{proof}
Let $L= \displaystyle\sum_{i=0}^{\deg(L)} a_i x^i$. Then
\begin{equation}\label{eq1}
\left(L \bullet \widetilde L /\widetilde M \right)^q=\left(\widetilde M^{\deg(L)}L\left(\widetilde L/\widetilde M \right)\right)^q=\sum_{i=0}^{\deg(L)} a_i^q ({\widetilde L}^q)^i ({\widetilde M^q})^{\deg (L)-i}.
\end{equation}
We now assume that  $x \in \mu_{q+1}$. Since $L\sim_{\beta} M$, we have that $M= \displaystyle\beta^{-1}\sum_{i=0}^{\deg(L)} a_i^q x^{\deg(L) -i}$,
and so
\begin{equation}\label{eq2}
M\left(\widetilde L/\widetilde M \right)= \displaystyle\beta^{-1}\sum_{i=0}^{\deg(L)} a_i^q \left( \widetilde L/\widetilde M \right)^{\deg(L) -i}
\quad\Longrightarrow\quad  \beta\widetilde M^{\deg(L)}M\left(\widetilde L/\widetilde M \right)=  \displaystyle \sum_{i=0}^{\deg(L)} a_i^q \widetilde L^{\deg(L) -i}\widetilde M^i.
\end{equation}
By substituting
$$\widetilde L^q = \gamma x^{-\deg (\widetilde L)}\widetilde M \quad\text{ and }\quad
\widetilde M^q= \gamma x^{-\deg(\widetilde L)}\widetilde L$$
into (\ref{eq1}) and then using (\ref{eq2}), it follows that
\begin{eqnarray*}
\left(L \bullet \widetilde L /\widetilde M\right)^q &=& \gamma^{\deg(L)}x^{-\deg(L)\deg(\widetilde L)}\sum_{i=0}^{\deg(L)} a_i^q  \widetilde L^{\deg (L)-i}
\widetilde M^i \\
&= &\beta\gamma^{\deg(L)} x^{-\deg(L)\deg(\widetilde L)}\widetilde M^{\deg(L)}M\left( \widetilde L/\widetilde M \right)\\
&= &\beta\gamma^{\deg(L)} x^{-\deg(L)\deg(\widetilde M)}M\bullet \widetilde L /\widetilde M.
\end{eqnarray*}
\end{proof}

\section{Permutation polynomials from bijections on $\mu_{q+1}$}\label{Sec3}

We start this section by extending Theorem~\ref{Zieve2013Theorem1.1}.

\begin{theorem}\label{prop:good}
Let $L,M \in \mathbb{F}_{q^2}[x]$ and $\beta\in \mu_{q+1}$ be such that  $L\sim_{\beta} M$. Let $k\geq 0$ be an integer. Then $x^{\deg(L)+k(q+1)}M(x^{q-1})$ permutes $\mathbb{F}_{q^2}$ if and only if $(\deg(L)+2k,q-1)=1$ and $L/M$ permutes $\mu_{q+1}$.
\end{theorem}
\proof
We apply Theorem~\ref{ppcriterion}. First, we notice that the conditions $(\deg(L)+k(q+1),q-1)=1$ and $(\deg(L)+2k,q-1)=1$ are equivalent.
For $x \in \mu_{q+1}$, we obtain that
$$x^{\deg(L)+k(q+1)}M^{q-1} = x^{\deg(L)}M^q/M=\beta x^{\deg(L)}x^{-\deg(L)}L/M=\beta L/M.$$ Since $\beta \in \mu_{q+1}$, the claim follows.
\endproof

A similar criterion involving self-associated polynomials is the following.

\begin{theorem}\label{var1}
Let $L \in \mathbb{F}_{q^2}[x]$, $\beta\in \mu_{q+1}$ and $t\in\mathbb N$ be such that  $L\sim_{\beta,t} L$. Let $k,s\geq 0$ be integers. Then $x^{s+k(q+1)}L(x^{q-1})$ permutes $\mathbb{F}_{q^2}$ if and only if $(s-t,q+1)=1$, $(s+k(q+1),q-1)=1$, and $L$ has no roots in $\mu_{q+1}$.
\end{theorem}
\proof
We apply Theorem~\ref{ppcriterion}. For $x \in \mu_{q+1}$, we obtain that
$$x^{s+k(q+1)}L^{q-1} = x^{s}L^q/L=\beta x^{s-t}.$$
Since $\beta \in \mu_{q+1}$ and $(s-t,q+1)=1$, the claim follows.
\endproof

Now we will show how we can derive Theorem~\ref{Zieve2013Theorem1.1} from Theorem~\ref{prop:good}.

\bigskip

\textit{Proof of Theorem~\ref{Zieve2013Theorem1.1}.}
Let $\beta\in \mu_{q+1}$ and $\gamma \in \mathbb{F}_{q^2}^*\setminus \mu_{q+1}$. Suppose that $(n+2k,q-1) = 1$ and $(n,q+1)=1$. Consider $L(x)=x-\gamma^q\beta$  and $M(x)=\gamma x-\beta$. Clearly $L \sim_{-1/\beta}M$. By Lemma~\ref{assoc_prop},
it follows that $M^n\sim_{(-1/\beta)^n}L^n$. By Proposition~\ref{Zieve2013,Lemm A_2.1},  the rational function $L/M$ permutes $\mu_{q+1}$, and so does $M^n/L^n$ since $(n,q+1)=1$. On the other hand, if $\ell(x)=\gamma x-\gamma^{1-q}$ and $\widetilde\ell(x)=x-\gamma$ then $\ell \sim_{-\gamma^{q-1}}\widetilde\ell $. By Proposition~\ref{prop:composition}, we have that $\ell \bullet M^n/L^n\sim_{-\gamma^{q-1}(-1/\beta)^n}\widetilde\ell\bullet M^n/L^n$.  Moreover, Proposition~\ref{Zieve2013,Lemm A_2.1} implies that  $\ell /\widetilde\ell$ permutes $\mu_{q+1}$, and so does the rational function $$\dfrac{\ell \bullet M^n/L^n}{\widetilde\ell\bullet M^n/L^n} = \dfrac{\ell\left(M^n/L^n\right)}{\widetilde\ell\left(M^n/L^n\right)}.$$
Therefore, by Theorem~\ref{prop:good}, the polynomial
\begin{eqnarray*}
x^{\deg(\widetilde\ell \bullet M^n/L^n)+k(q+1)}(\widetilde\ell \bullet M^n/L^n)(x^{q-1}) &=& x^{n+k(q+1)}  \left(L^n\cdot\widetilde\ell\left(\dfrac{M^n}{L^n}\right)\right) (x^{q-1}) \\
&=& x^{n+k(q+1)} \left(M^n-\gamma L^n\right)(x^{q-1})\\
&=& x^{n+k(q+1)} ((\gamma x^{q-1}-\beta)^n-\gamma(x^{q-1}-\gamma^q\beta)^n)
\end{eqnarray*}
permutes $\mathbb F_{q^2}$. Conversely, if $x^{n+k(q+1)} ((\gamma x^{q-1}-\beta)^n)-\gamma(x^{q-1}-\gamma^q\beta)^n)$ permutes $\mathbb F_{q^2}$ then by Theorem~\ref{prop:good}  we obtain that $(n+2k,q-1) = 1$ and
$\ell\left(M^n/L^n\right)\Big /\widetilde\ell\left(M^n/L^n\right)$
permutes $\mu_{q+1}$.  Since $\ell/\widetilde\ell$ and $M/L$ permute $\mu_{q+1}$, the latter condition is true if and only if $(n,q+1)=1$.  \hfill$\square$

\bigskip

In the proof above, we use Zieve's characterization of degree-one rational functions permuting $\mu_{q+1}$.  For rational functions of higher degree satisfying the same property, an analogous result follows, providing another way of constructing permutation polynomials.

\begin{theorem}\label{var}
Let $L,M \in \mathbb{F}_{q^2}[x]$ and $\beta\in \mu_{q+1}$ be such that  $L\sim_{\beta} M$. Let $k\geq 0$ be an integer  and $\gamma \in \mathbb{F}_{q^2} \setminus \mu_{q+1}$.  Then
$$x^{n\deg(L)+k(q+1)}(L^n-\gamma M^{n})(x^{q-1})$$
permutes $\mathbb{F}_{q^2}$ if and only if $(\deg(L)+2k,q-1)=1$, $(n,q+1)=1$ and $L/M$ permutes $\mu_{q+1}$.
\end{theorem}

This result can be proved by replicating our proof of Theorem~\ref{Zieve2013Theorem1.1} without specifying $L$ and $M$.

Throughout the remaining of this section, we will mention the polynomials satisfying our permuting conditions  several times.  To ease our reference, we call them \emph{good polynomials}.

\begin{definition}\label{Definition:good}
Let  $L, M \in \mathbb{F}_{q^2}[x]$ and $\beta\in \mu_{q+1}$,  and let $k,s\geq 0$ and $t>0$ be integers. We say that the pair $(L,M)$ is \emph{$(\beta,k)$-good}   if $L\sim_{\beta} M$, $(\deg(L)+2k,q-1)=1$ and $L/M$ permutes $\mu_{q+1}$.  
We say that  $L$ is \emph{$(\beta,t,k,s)$-good}   if $L\sim_{\beta,t} L$, $(s-t,q+1)=1$, $(s+k(q+1),q-1)=1$,  and $L$ has no roots in $\mu_{q+1}$.
\end{definition}

Because the conditions $L\sim_\beta M$ and $M\sim_\beta L$ are equivalent, we can reverse the roles of $L$ and $M$ in the polynomials that appear in Theorems~\ref{prop:good} and~\ref{var}.  This means that, whenever we have a pair $(L,M)$ of $(\beta,k)$-good polynomials, we assume that the following four classes of permutation polynomials over $\mathbb F_{q^2}$ are automatically
produced:
\begin{itemize}
\item $x^{\deg(L)+k(q+1)}L(x^{q-1})$
\item $x^{\deg(L)+k(q+1)}M(x^{q-1})$
\item $x^{n\deg(L)+k(q+1)}(L^n-\gamma M^{n})(x^{q-1})$
\item $x^{n\deg(L)+k(q+1)}(M^n-\gamma L^{n})(x^{q-1})$
\end {itemize}
where  $n\in\{n\in\mathbb{N}\mid (n,q+1)=1\}$ and $\gamma\in\mathbb{F}_{q^2} \setminus \mu_{q+1}$.
Similarly, whenever we have a polynomial $L$ that is $(\beta,t,k,s)$-good, we assume that the following class of permutation polynomials
over $\mathbb F_{q^2}$ is produced:
\begin{itemize}
\item $x^{s+k(q+1)}L(x^{q-1}).$
\end{itemize}

Proposition~\ref{Zieve2013,Lemm A_2.1} provides a complete description of all good polynomials of degree one, namely
$(x-\gamma^q\beta,\gamma x-\beta)$ is $(-\beta^{-1},k)$-good whenever $(2k+1,q-1)=1$, $\beta\in\mu_{q+1}$ and $\gamma\in\mathbb{F}_{q^2}^*\setminus\mu_{q+1}$.  The fourth class of permutation polynomials associated to this pair is the one  obtained by Zieve in Theorem~\ref{Zieve2013Theorem1.1}.

We can also precisely describe the pairs of good polynomials of degree two.



\begin{proposition}\label{prop:grado2}
Let $L$ and $M$ be polynomials of degree two over $\mathbb F_{q^2}$. Then
$(L,M)$ is $(\beta,k)$-good if and only if $q$ is even, $(k+1,q-1)=1$,
$$L(x)= (C_1+(i+1)C_2)x^2+A_1+(i+1)A_2 \quad\text{ and }\quad
M(x) =(A_1+iA_2)x^2+C_1+iC_2
$$
where each $A_j,C_j\in \mathbb{F}_q$, $\xi \in  \mathbb{F}_{q}$ and $i\in \mathbb{F}_{q^2}$ are such that $Tr_{\mathbb{F}_{q}|\mathbb{F}_2}(\xi)=1$, $i^2=i+\xi$,
$$
(A_1+iA_2)(C_1+iC_2)\neq 0, \qquad \text{and} \qquad
A_1^2 + A_1 A_2 + C_1^2 + C_1 C_2 + \xi ( A_2^2+ C_2^2) \neq 0.
$$
\end{proposition}
\proof
Let $L(x)=ax^2+bx+c$ and $M(x)=a^{\prime}x^2+b^{\prime}x+c^{\prime}$ with $a,b,c,a^{\prime},b^{\prime},c^{\prime}\in \mathbb{F}_{q^2}$ and $a, a'\neq 0$. Suppose that $(L,M)$ is $(\beta,k)$-good. Since  $(2(k+1),q-1)=1$, we must have that $q$ is even. Since $L$ and $M$ are $\beta$-associated,
we have that for all $x\in \mu_{q+1}$
$$L^q=x^{-2}(a^q+b^qx+c^q x^2) =\beta x^{-2}(a^{\prime}x^2+b^{\prime}x+c^{\prime})
\quad\Longrightarrow\quad (a^q,b^q,c^q)=\beta(c^{\prime},b^{\prime},a^{\prime}).$$
Since $i^q=i+1$, we can assume that
$$\begin{array}{lll}
L(x)&=& (C_1+(i+1)C_2)x^2+(B_1+(i+1)B_2)x+A_1+(i+1)A_2\quad\text{ and}\\
M(x)&=&\beta((A_1+iA_2)x^2+(B_1+iB_2)x+C_1+iC_2)
\end{array}$$
for some $A_i,B_i,C_i \in \mathbb{F}_q$.

To show that $L/M$ permutes $\mu_{q+1}$, we consider the curve $\mathcal{C}_M$ associated with $M$ given by $(L(x)M(y)-L(y)M(x))/(x-y)=0$, that is,
$$\alpha_1 xy + \alpha_2(x+y) + \alpha_1^q=0$$
where
$$\alpha_1=A_1 B_1 + A_1B_2+B_1C_1+B_1C_2+i(A_1B_2+A_2B_1+B_1C_2+B_2C_1)+\xi (A_2B_2+B_2C_2)$$ and
$$\alpha_2=A_1^2 +  A_1  A_2 + C_1^2 + C_1 C_2 + \xi( A_2^2 + C_2^2 ).$$
We note that $\alpha_2^q=\alpha_2$. If $\alpha_1=0$ and $\alpha_2 \neq 0$, then $x=y$, and so $L/M$ permutes $\mu_{q+1}$.
On the other hand, if $\alpha_1\neq 0$ then for $x\in \mu_{q+1}$ we have
$$y=\frac{\alpha_2 x+\alpha_1^q}{\alpha_1 x +\alpha_2}\quad \text{ where } \quad y^q=\frac{\alpha_2 +\alpha_1 x}{\alpha_1^q  +\alpha_2 x}=\frac{1}{y},$$
which means that the curve $\mathcal{C}_M$ has points in $\mu_{q+1}^2$. Consequently, $L/M$ does not permute $\mu_{q+1}$.

Suppose that $\alpha_1=0$ and $\alpha_2\neq 0$. If $B_2\neq 0$, then from
$$A_2B_2+B_2C_2=0 \text{ and } A_1B_2+A_2B_1+B_1C_2+B_2C_1=0 $$ we have that $A_2=C_2$ and  $A_1=C_1$, and thus, $\alpha_2=0$, a contradiction. A similar argument holds if $B_2=0$ and $B_1\neq 0$.
So $B_1=B_2=0$. In this case, the polynomial
$$M(x)=\beta((A_1+iA_2)x^2+C_1+iC_2)$$
has a unique root, which belongs to $\mu_{q+1}$ if and only if
$$\left(\frac{C_1+iC_2}{A_1+iA_2}\right)^q =\frac{A_1+iA_2}{C_1+iC_2}.$$
This is equivalent to $(C_1+iC_2+C_2)(C_1+iC_2)=(A_1+iA_2+A_2)(A_1+iA_2)$, that is, $\alpha_2=0$, a contradiction.
Hence $M(x)\neq 0$ for every $x\in\mu_{q+1}$.
\endproof

In the next results we  show more examples of good polynomials.

\begin{proposition}\label{prop:grado3}
Let $q$ be odd, $k\geq 0$ be an integer and $i\in \mathbb{F}_{q^2}$ be such that $3\nmid (q+1)$, $(3+2k,q-1)=1$ and  $i^q=-i $.
If
\begin{align*}
&L(x) =-A_1x^3+(-3 A_1 + B_1-i B_2) x^2+(B_1-i B_2) x+A_1\text{ and }\\
&M(x) =A_1x^3+(B_1+i B_2) x^2+(-3 A_1 + B_1+i B_2) x-A_1
\end{align*}
with $A_1,B_1, B_2 \in \mathbb{F}_q$ and $A_1(3A_1 - 2 B_1)\neq 0$, then  $(L,M)$ is $(1,k)$-good.
\end{proposition}
\proof
By the condition  $i^q=-i$, we get that $L \sim_1 M$.
The function $L/M$ permutes $\mu_{q+1}$ since the curve given by $(L(x)M(y)-L(y)M(x))/(x-y)=0$ factors as
$$A_1(3A_1 - 2 B_1)(x y + y + 1)(x y + x + 1)=0.$$
The component  $x y + y + 1=0$ has points in $\mu_{q+1}^2$ if and only if $T^2+T+1$ has roots in $\mu_{q+1}$. This cannot happen since $3\nmid (q+1)$.
Therefore, the polynomial $M$ has  as a root $\xi$ in $\mu_{q+1}$ if and only if $M(\xi)=L(\xi)$, which is equivalent to $A_1^3(3A_1 - 2 B_1)^3=0$, a contradiction.
\endproof


\begin{ex}
 In the proofs of \textup{\cite[Theorems 3.4 and 3.6]{BQ2017}} there are reciprocal polynomials $L$ and $M$ of degree three over $\mathbb{F}_{q^2}$ such that $(L,M)$ is $(\beta,k)$-good, $(3+2k,q-1)=1$ and $L$ has no root in $\mu_{q+1}$.
\end{ex}

\begin{proposition}\label{anydeg}
Let $q$ be even, and $k,s\ge 0$ be  integers such that $(s-q,q+1)=1$ and $(s+k(q+1),q-1)=1$. Consider $$L(x)=\sum_{i=0}^{\lfloor(q-1)/2\rfloor} (a_i x^{i} + a_i^q x^{q-i})$$ in $\mathbb{F}_{q^2}[x]$ such that $L(x)\ne 0$ for any $x\in\mu_{q+1}$.
Then $L$ is $(1,q,k,s)$-good.
\end{proposition}
\proof The requirements,  $(s-q,q+1)=1$ and $(s+k(q+1),q-1)=1$, immediately yield that $q$ is even. For $x \in \mu_{q+1}$, we compute
$$
L^q = \sum_{i=0}^{\lfloor(q-1)/2\rfloor}(a_i^q x^{qi}+ a_i x^{q(q-i)})
= x^{-q}\sum_{i=0}^{\lfloor(q-1)/2\rfloor} (a_i^q x^{q-i}+ a_i x^i)= x^{-q} L.
$$
Hence $L\sim_{1,q} L$ and the claim follows.
\endproof



\begin{ex}\label{Ex2}
Let $q$ be even. Consider $h(x) = a x^i+a^qx^{q-i}$ over $\mathbb{F}_{q^2}$ with $(2i+1)\mid (q+1)$ and $a^{(q^2-1)/(2i+1)}\neq 1$. A nonzero root $x$ of $h$  satisfies $a^{q-1}x^{q-2i}=1$. It belongs to $\mu_{q+1}$ if and only if
$x^{2i+1}= x^{2i-q}$ or $x^{2i+1}=a^{q-1}$. Hence
$$1=x^{q+1}=\left(x^{2i+1}\right)^{\frac{q+1}{2i+1}}=\left(a^{q-1}\right)^{\frac{q+1}{2i+1}}= a^\frac{q^2-1}{2i+1}\neq 1,$$ a contradiction.
Under the $\gcd$ requirements from Proposition~\ref{anydeg},  $(s-q,q+1)=1$ and $(s+k(q+1),q-1)=1$, we conclude that $h$ is $(1,q,k,s)$-good.
\end{ex}

\section{Permutation polynomials from bijections on $\mu_{q^{k-1}+\cdots +q + 1}$}\label{Sec4}

There is a natural way to extend Theorems ~\ref{prop:good}  and \ref{var1} for $k\geq 2$. We write $q^k-1=(q-1)(q^{k-1}+\cdots +q + 1)$ and consider $\mu_{q^{k-1}+\cdots+q+1} =\{ x \in \mathbb{F}_{q^{k}} \mid x^{q^{k-1}+\cdots+q+1}=1 \}.$

\begin{theorem}\label{general}
Let  $k \geq 2$ be an integer.  Suppose that $L,M \in \mathbb{F}_{q^k}[x]$ and  $\beta\in \mu_{q^{k-1}+\cdots+q+1}$ are such that
$M^{q}=\beta x^{-\deg(L)} L$ for any $x \in \mu_{q^{k-1}+\cdots+q+1}$. Then $x^{\deg(L)+k(q^{k-1}+\cdots+q+1)}M(x^{q-1})$ permutes $\mathbb F_{q^k}$ if and only if $(\deg(L)+k(q^{k-1}+\cdots+q+1),q-1)=1$ and $L/M$ permutes $\mu_{q^{k-1}+\cdots+q+1}$.
\end{theorem}
\proof
We apply Theorem~\ref{ppcriterion}.
For $x \in \mu_{q^{k-1}+\cdots+q+1}$, we obtain that
$$x^{\deg(L)+k(q^{k-1}+\cdots+q+1)}M^{q-1} = x^{\deg(L)}M^q/M=\beta x^{\deg(L)}x^{-\deg(L)}L/M=\beta L/M.$$
The claim now follows as $\beta \in \mu_{q^{k-1}+\cdots+q+1}$.
\endproof


\begin{theorem}\label{var2}
Let  $k \geq 2$, $s\geq0$ be  integers,  $\beta\in \mu_{q^{k-1}+\cdots+q+1}$, and $L \in \mathbb{F}_{q^k}[x]$ be such that  $L^q=\beta x^{-t} L$ for some fixed integer $t$. Then $x^{s+k(q^{k-1}+q^{k-2}+\cdots+q+1)}L(x^{q-1})$ permutes $\mathbb{F}_{q^k}$ if and only if $(s-t,q^{k-1}+\cdots+q+1)=1$, $(s+k(q+1),q-1)=1$, and $L$ has no roots in $\mu_{q^{k-1}+\cdots+q+1}$.
\end{theorem}
\proof
We apply Theorem~\ref{ppcriterion}. For $x \in \mu_{q^{k-1}+\cdots+q+1}$, we obtain that
$$x^{s+k(q^{k-1}+\cdots+q+1)}L^{q-1} = x^{s}L^q/L=\beta x^{s-t}.$$
Since $\beta \in \mu_{q^{k-1}+\cdots+q+1}$ and $(s-t,q^{k-1}+\cdots+q+1)=1$, the claim follows.
\endproof

Next we show a class of polynomials that can be used to construct permutation polynomials using Theorem~\ref{var2}.

\begin{proposition}\label{Prop:k=3}
Let  $0<j\leq k\leq q/3$. For $A\in \mathbb{F}_{q^3}$ and $B\in \mathbb{F}_q$, the polynomial
$$L_{A,B}(x)=Ax^{3kq+3j}+Bx^{(k+j)q+2j-k}+A^qx^{3jq+3j-3k}+A^{q^2}$$
satisfies $L_{A,B}^q=x^{-3kq-3j} L_{A,B}$ for every $x \in \mu_{q^2+q+1}$.
\end{proposition}
\proof
We observe  that $1 < (k+j)q+2j-k, 3jq+3j-3k < 3kq+3j\leq q^2+q$.For $x \in \mu_{q^2+q+1}$, we have that
\begin{eqnarray*}
L_{A,B}^q &=& A^qx^{3kq^2+3jq}+Bx^{(k+j)q^2+(2j-k)q}+A^{q^2}x^{3jq^2+(3j-3k)q}+A \\
&=&\frac{A^qx^{3kq^2+3jq+3kq+3j}+Bx^{(k+j)q^2+(2j-k)q+3kq+3j}+A^{q^2}x^{3jq^2+(3j-3k)q+3kq+3j}+Ax^{3kq+3j}}{x^{3kq+3j}} \\
&=&\frac{A^qx^{3jq+3j-3k}+Bx^{(k+j)q+2j-k}+A^{q^2}+Ax^{3kq+3j}}{x^{3kq+3j}}=\frac{L_{A,B}}{x^{3kq+3j}}.
\end{eqnarray*}
\endproof

Consider the following number $R$.

{\tiny
\begin{longtable}{c}
$R=$5441\\
5676365797077637141450586653222239320299764938998099705894071352243683085327110128286703483238476973397002404\\

6336959383040020572535656462087700990144448533299525758921522351819247209622176668910816961864549530309479655\\
    8944772731555018237185037053331470731435723299754298878959745419466062019178032287477005478718395309071239879\\
    9512272244940182092323063835500724103308099775996494047115010379686917413489292225619989289985558426384102396\\
    0242491499439231026956456844274382784161805452560664946999045428509995577306871121110395588039150140763221275\\
    8254726926983058860397893644379710130686641078337368194371747751808789029762359209804585376900978052798695554\\
    9717105377303966806014138301988303903745229199508482837237155735526437742147986415584307271942779370464042363\\
    4291632504052005616026693998097888408809903568049767905049020272041868743359931306317006104083213066173082930\\
    6547227028646224916826046697643250240184194946963570860743922536593958949634787363676767604817788439890772475\\
    7196520102634540337674729327883553726822895205403422554082727098547413192676902596050860718136721876365947253\\
    6719476168404542519163512021540837073195835107996580808265799380285437608944664764663025601850714240587122556\\
    6431259364390820765272488409000721487274813407479996178887406683776378602708245543719485677456061038242819904\\
    8969881205063842269154691565470178975516788962990716343152188005276428528237302954758849529071129046029630601\\
    2701325870360960933979759946540674382336022436823131958732148421091329532463627675691466831201845431885134658\\
    3474173657897011457015283289320304453713008035355987980917917016166872401884009076569961776946686104911869779\\
    0645651903080948306065535393677153283039700127012609755241629454579044093292678121528699089867030766409525749\\
    2286363707063912198790921211416716743641296977397383441494891453337913803716999605011616643879015822925136864\\
    3174346389465592919537305423464081034051482989338742834208094814564187745474731308510156871313398631836997710\\
    6320444622491960037709440548998233256599246697876822834793085410714547334077801872546103718595830615096888312\\
    5590548388716551373976296904376655228167444469799840919329902226478698668061788498911871968904140805748099872\\
    5092577098839437028839516642873714911317433662425017169406954870413442982643094382849026978037312441649951096\\
    9695178663475056117305629449725158578810877922500872312479344149054119555113940616046236617452961213592876876\\
    8103738719543975199278488746433478738410141239554625448441884255142024258531235926825328566967005612193046894\\
    8220152110777223032235849218587634764953849659031033301063504002326230300880139023469084387583024223394899189\\
    2688937450628774612107671933525346928580016163482833343334391253028082979785505797065431793882742029537684772\\
    5786263340403554434189660391874946819314380154683484739982862469991782471944972525248814062972408602600811154\\
    9884102202253383927714575817913570588613734470225891229318737324882350450825940131197628289302128762241548533\\
    8733147304316334370149507205975805061127579825575571617277440699560375071576867014227409481740347534439471100\\
    6203501165020588848877494851085221055109230470915533334073332873889306632923379303416240611321578785371081217\\
    1064076204520660765340892390794016569450447697366180324844094481748696014468725850230575433303712319077426059\\
    0194451222928976562084082101826372855649842925376937549435353840461071482489758019275067372745719883672588194\\
    0100672337049199097820906764832552036148970777808523180228351064955117345460277792810746801524160653864618983\\
    4337948430373685281741112506353180948763994508488951805348767798321262825290396359180119714479785513465900591\\
    7676725673195406402044738415237326183314226982646375043639054102846739631047484710558807459137190629486553263\\
    3682047674823464125548159054729808349791123139234637133532156038831544634205103503788620810915861120140001963\\
    1584024696369239876060359569279600736079363924979215359725649714781463252767724351718722681105102699276375324\\
    0384506247732338356021665611353327695267317862325770468749664904657967438733765479781622116932693334695677063\\
    5021205458687762701406349467620957817050413973635581810943606941670515524356485356571155740943898475421702809\\
    5328676008858976687822051383616629525783816610181607621664444979682725068595675237124932299708264808363528942\\
    0881448719330607429232987519041514099309662073557815006801007281373889310847641642859115421530144682582679087\\
    7704122351009188338945882237335685795239318377735088786938767247579794678803655234167516463130920845119715886\\
    1629372381818703314602304582053238969909674005980237110370063065319091822028589889211550783393412843377688225\\
    8259724711940968819253213857618049433082085183888609617589424587428398173100129687015187077690094944334828587\\
    2213002122928439661351711328209089339184016628978532673325879487978992297152378010478764657897256419801633437\\
    1191044029613777953398751698044971600424726889865075486667173790922976006634771710588891927968794784144642864\\
    0390361604745677735521915767160461155055726724680510417700985842245836701287078904459538130339710091582259422\\
    8854716508930449972389901715233901487226760880411435477734570358653676866482732143277138904893831788893482839\\
    8921704512034212740346203802309390227561424016315394336234930255229433906591437671357190277680112841734559224\\
    1061910219751200817622942287352522311334866074635676611790473798201395797694275652786481296475644510559592192\\
    5037860742127073978025713649127243100396941997934575230130759227211571857160592586971459591625321386075752952\\
    5544003583619000264703750680594854001448684682253399194158162164140451208798493385642564083380842659897312005\\
    3427210834537418677982977482186537995497695893831499050923162179304307398023044705293108189762496851692286412\\
    4542226785468585982758427794709522051985923218747787443933713380846371828374267034110361582105434445000107569\\
    7438512837097132759601051646936589845662309049274208250980265523794484372040768359655580591345562440380824641\\
    0555056100280634388850010559180534865615782040876613267184117826865359280973798164069174332239401502320843972\\
    0986490223261354824585537306076374143416480611047679622657325551992713203141841305913422787411508732810935943\\
    4099750075428932812216653454658943946152870537430456821805938197644782060254771056646975453519489660129170104\\
    9898265563420924817580790969390323469130003027957909353214386197982662082638608993159854578745426325002908736\\
    2354245942532329081241500134580749123106728764064079704232697614431982511373366042541057826985156167175610985\\
    5070279608393085353244064289317455900186428315023110772051461316286747252722283519899171835356319848937347312\\
    4318004900726204880833304677992173016897699726555630987195460587064880907045467298822393702017265882230319485\\
    4394770202237203884513509086524170646176746440689103820830267668970433966755216656727515566421372414462458733\\
    8434619732683278713362121933780651748772850991377493383980253592430061579343736847572995877169058890586459775\\
    0986722826639148526561660985651553673030485122801346167009245347434068804774796975838320677459988296108574280\\
    7711743575124944565258925741713531135688545896909650947896299287804044388058649499477734164166651132883018182\\
    7205957659186898447627033492019022767168172854802670625197079589468800919191194761421078679326513462393290089\\
    2639078856778542881708391822549842685628423583972580177918567091078826426763566941090746958028388721372314183\\
    5483024610678249708343793331328805041034749305031586884583078470875357211307871560884412140189387310458668024\\
    3725826730618059717310209671658945961201346064610156034490778245504504780087514532289326026078368350930352554\\
    2405951269777800116693154508932001606247340608601093274226101665837624606790771250033711014893945674496640447\\
    2012366425461908925285156607679796310892001962376469081066796453336396417222512862287261191364829346075360124\\
    7338729163524500533038072457698416303779261070622108918922138797178137880355644016125150577082980172203706013\\
    5689911773905446866787479890004708894835849731075700710709776728343798070730903589488825848747326158288250287\\
    3351437970551535656281749495270860409400865248288721051623572064345870404800519019299425156291196761358145354\\
    5170457271892655560997264386516345553684347663166042864172722172403677201047845701953577437940108878949604731\\
    2177991861785452882788240837676067407286825800433023930929911201392485164262815359182481802294790534060035416\\
    7424152934464274403996109614275383018031942010510110368207665799069073863719678680861925809798772709679551999\\
    4744347138701312009631885528265078802413619180039369306198123969046972215321863521329632645629362260463254280\\
    3145858269784476443612533592493615276908411866122991092331055119089870642103165519412293192030211497400476389\\
    7401931055697353915473697541464311064390315313474962233825653272796591469317437466812324513495575413412820477\\
    0449204575490326151571630803661655104267398461641657604821947727132097305756720974094675886948956263455030014\\
    1941185500708532443796337658208985229232308951041649019077255116796761990492069506645888710340189787706571162\\
    2842081300294443993348693306393786169064549490459878168313163366539588317729738837679613733675123846023804014\\
    3546716680423164986514316318419783274273599751563135214747688528490994035302257059307664180598830881458114869\\
    8407763661135734226970376857671681604815917310984273777977466589162692986238889075939297176220672439861314948\\
    8424169127650550087685504671833228800754719107396574684095090757030590956144811704113197610081283267701051340\\
    1542027279774528942622159593802881601102141623072904220445506381558604401853316676290227079284317616227539032\\
    4806509514581872945268177762493910525322503516841695869327632862900408912632577002634479460675665026602232572\\
    1413867050246320460295143763254301785435859728562383630682408375788828810018100538968641327644107234855637405\\
    1299505766710339663508615570070019716765404076025586134289830765268245969905786068713700274264234574740987976\\
    6357761447017190976608827084261547252601508674730918195638221580756809308343302702858039964113589953873168879\\
    6405807890573578659413803215626638616196443278921723799507463247989428698099794244336878240633234363776733955\\
    9518056160816143879716994299180607651707258668138331525218931503941332021627838295278504441097972942177497405\\
    0489024243402497604556533928560175811720437356100403078839480474376403021955175735701365025160398160999719831\\
    2931568432705460376985683217311501662192424988498341236881373817189771766287490829346627253129265632494726241\\
    9511759883537921065326449403975960182969781632351808300342577183122791916986914200391694999434708754521011256\\
    9799778824184473368480174030793959013265039421344719477748921365327272482477175026638209899110588371944652232\\
    9599874154721868553713695919667030628398906038052074222263267285013991372961837880439233936891465017034175331\\
    7258828287622272256649089178955743276102485576354658085186604080021292057138799441079140560811788563405086616\\
    6833497853183625550499970849579167398895802685384198433193834389084182981882957594334094477398407106673453012\\
    9987938208339397930696014845339294785494336208747000551132632644415520480733391207776878049481640527287345831\\
    6796932234066567348023789182310425860546263626806655043719784282164773224883489300170825670909897154762224520\\
    3443378769285901394475717147937341961932072256599661709604126168324549503402382324464036091705187775494226667\\
    7795512598589531055905285647652992819988971411518359187385511258712718723696134866854173533214174417997920252\\
    9471427169325065816894169694483171743047555453252472262232020275216589553182233964790014933010127533409422975\\
    4052267816664609308904924758492429336034187262515967829467044710704646999724287390961926451699352139966941922\\
    6292326374643725951555137540306793977580271045089606973285060220466860449604188134759204351492965925660851812\\
    9546232679058898783908413835826174153697035888676261214785212423766278942280788862744491741955045748380069394\\
    5142175038337110114804708952364102251668933733348080704388212861611189997375301969671138839231048984069865239\\
    8688122662914342676874382922451014920871164653395763937978117268964779220821030908457907007195231655806206254\\
    6863968067429848332890465157583974726299207021768157927473743086887696497121895738490447379939709843149889563\\
    5729569392425153233465351439334826629005804315761937056188907507933229907177031414100904526737160129389911422\\
    6692642752121758089982428307379923871101096707634282296484901440918768771846939087197576349549322883549852270\\
    0557069555989889836436708969834788207363701890666951076722021276709709497365287871478462244499242339347207994\\
    7607458238267257772863231626794456273956846404569278553301572733496108328059880633137153090101507168483800090\\
    8938480903729477865026130023108261903000169368084230108841126430160023583762092716546387582558659189497268952\\
    7014488816300876256671631262191070575629193284986210308296661955092737727930729484798551813428431614726082142\\
    7073915093155450702487541186910993134617165501890181613923533136545381341275748284897854986347640058427738089\\
    1291149553769352080951989427220600453078805201312190196274139628166326214227253023257937457964554802318561638\\
    1358465791282608720038205655613763051084804940479241223517287858831930167237181950413289506081013617185311670\\
    5498164579803758047365230909715630303463203273459065547897101631109941726288523352549265598384903704366371448\\
    0529491496719703617593204656211563764619556106936971896551222995748655220406522350497250238320344336108331917\\
    5247437750773933649095058216169688851967572996370472868518363507608160876048913807547791203420944956755619145\\
    1313150978443872190362189419878684586643585845595618770154177037028116332948497008725408402016916288662862893\\
    0739311569049045108961327683191628547961011020525788180897254987870172578770756927676013405891766246967858782\\
    8627224822443706883774067727687713945158424949381090647038854036823270319585116054285326689859793959425939724\\
    8438136932174118446137397991761393731551897112560124269393320811749587167847956682889557694317302751133850113\\
    8298747120718441477582389736954334210243016583549332012582270905745887176879751423602253869054812889940438487\\
    0388422239034236587890715879527894802885420360088070773249052217152053685142177330972211566356457507645711726\\
    3721073265110386145603830016114546857025453418812242851881225997717191974840523435334247814434837887643660857\\
    9682939203946121331243644002689556199545385687000684133106920232577332686571633713819645979045406013039488482\\
    3147022253685662706680833640719439459088056745436349578845626028221682150541613830524512080398784766308979913\\
    5265086807837092138668182051489211272882001922267060484626038007881783610156360271271712808803588678384169126\\
    7005192667592259518872257999078031080025204572798046374356115137373379928150159004928179317904303106187985009\\
    5037831548697133283663428896145043503351608739361799081938310024795881499906964238304555046059542416782642340\\
    8104900385421823403722310881582749572213500529196906385190640818379975695535592238609851314598821682470996359\\
    2229286865545525575595017536347784865227429267248369484643872695431956204338232000142127235023640989601903121\\
    9944955933686587678590180924627665532282196428078739791556691107136661312302926448505019055285749704407131676\\
    5193057521489403702987778775710127816954626675948567701000821739379119326922270948767670430870588277870204368\\
    9657368273130637406254843365202903268348057355201728144661865318938456781978837363131951227858814558925565120\\
    7843757910169902851012643487560100781178613399212472287256348170686338034890567943232234384549446873382894893\\
    2866232488684222240299585664726740629438741309065052649022373310828199292838953207050614876751381914989363636\\
    3391692517957600183435930077485093405538852422992902062931703852569701402972619485386652994643719390594215561\\
    6310325316109364129429956185033575730154791596707099146521546143490343602569824590619789235703626306055497549\\
    2388645719852664174071118433087751989570061121913149185955249260499273739911466731990153936027953803742410995\\
    6631235907152814953149389089617813638471207106393389699173874005347391233821449679423578320219040755685728983\\
    6769755096906499581845041348851273381829024573408537829172738281763401144592224223713291274118913854380289097\\
    0077290505729627312772227044709500842415153327489482022642867939868214689918289962095959964629107470359482998\\
    3407857124816371903318224550139729695016277210842990660746264963282740764595207917995917918484622570584627348\\
    5006098120080659744231323634560378964616745116983187183736087632891018149450945401979583255354498237377711680\\
    2671247828518777460018673953532671952182754802907041393435646725079792861526339776744157270165952603592885323\\
    8471481124630551154205889026528054965029764048504200889447497635347615325716431975536392144109014551118727501\\
    3691765648081394375865996000820638796842761345460800333739691751254235742537841459173185552777780093538125276\\
    7616434912390335833370896275147203260532732196462200484975248389957007837488033548793555103089465045028446899\\
    6512975079654347947866868137772377106752307767910043187200629656518377089393573388224716893207521584923382358\\
    3734873952596013593261368108525516931849995384182826189274394790977163461078276150136569070829327277115568417\\
    7200170368657514688251322304079502076370936037771884999382391749377303348809734837465373608193262760012418218\\
    1774520706654620119680305954399171623350208440475614873113097504538644510034498772026022834356426457411679722\\
    2919742875391160172487149159309746377740073580970025616636581415412271087856284352702404565728540646232775507\\
    5960769443763193578052264643718304636332949072948292517532450888059221301096638954585526903775046290804056972\\
    5524148026141673973101724486877630055105702267474486069784337443160350005250345481183412288786127036469230041\\
    3770220285449362347474070751525507529354023941782808128343582796921033348125578558592057424256311323046659043\\
    0597630553812986399507539676954764441603712170944769002686285096209510512607362134550399999056992339479774660\\
    1777362679119869166071171143465389815823934703181607119209230600484964862090339166555016922644266979490346760\\
    4705328708861354936380988024324607615465735246046895905408277023296783032917432101019784186947169751242709541\\
    7898208545924831405013110245964751164953446414255965182473415463200853095089195975653389259607493537514313891\\
    6216595568693681959655373467746221293373685097121149021673811801542091319586172204953515963274062088698527309\\
    9045570404392447494178766524123030111014926595037346903475794541920205195743602674736821412305622108400325951\\
    9878113555268726327839409689296551118517251421452232346854122904215296481789565496987596547127557838653186882\\
    3044492173913302767188204167903803930577822241774657950641484747449496045190023518957044034513955585582433481\\
    9797069635112208802726743956221781157141282009727459789638549563801221000317957089270314037786876931860468235\\
    2752444541551597246021942758013413334987287862977524743019672068380415261253600556864464354729463764013652948\\
    5994049744660524915490097084740230607901011726592502029193096754172984496239830183327855983614590280103233344\\
    8306097079103406307181271975013233440433364954887357912484201209037708596324274831198373966550269318587403765\\
    4645284197696058701974469727232308535639586216641851913635056738269828428280233669433204654578826157244410459\\
    8178482570110548435746348630346941314132786683912980853409042857084662233346478681888997093051506221649058388\\
    5163335210045730535848275976325438298772554869342084908529735890416200518055877065128393199889865590491984112\\
    2962531747580372461927624146319648822842336701623968671409825086823826466591604096341113147409723295681598759\\
    0932491059417625850045358415279299097189461140296565128890128152603446993456704436084877645882653653498102404\\
    3056484623387731382293927929110905722955483299141373981449715818654872612804452515317650876794480963308930890\\
    1350587651486029484226217615291817698707577749061815246815046891993914607522037037175970275639162458368863438\\
    2346701262281568404837196339502611398957631412450844236796342702797712219097679345565057611592406259083770380\\
    0778236349535270364531494839215708749121250171838192092790198087703883981845010802552297641213130719699800906\\
    3569075419836660364014575267982975236098416699776785978865202292442548716847259078916081425125824538606418934\\
    2721225599831795405593499860422364278999971027212842363625405158057109319923560401420407318697975834399286637\\
    2498305325475374393173956963574968458320086972229999145910324568419658653451594662604001829976922505870753859\\
    5359929868840957707850376312243479177448617613928208819743361580616121526456381668932220929137063010941337428\\
    8566346712977495634778091827349356515847860475397308388728214276471766620557939870322987208167260359236461856\\
    2540469797835303476629124954343724419351661537801260379810146062452808054136475643600087964320703164401994529\\
    5632944077653400383480676392172644005844552117547393127550444868155557054291816846069936359700218112621429828\\
    9810045365261797064019647540711030922465602885241109048848082204538940277596928121655552168933623289558870916\\
    5592826667409974328104601722405388349395372358524466126906817910131412631107332986133463524460976723833624687\\
    1796226582963052215176162174423629749766945639579970504790964880881060620424544510283732049287088616672728771\\
    2607513190726257891387293300611164045529774996580143947478820193079991185686076224369564827979675310605856416\\
    4712370035647828959274282719220617633739780209780147686431164573731590081870982425423334286722593183905441781\\
    6230181467826526881459263820122883674818602831915148474334872743876698975397708901068652552904930455544107569\\
    4671233464315628840647040987953942402315716402081777908661918899912326033777194183574588709483196180179236175\\
    6218738060218685425725092886049154611446816954463982709110142736026315202893479107973534155976592711615520257\\
    8236606295200792950675759732623770329693587330859819259084821667322617711635517774417923652850492073417640285\\
    4148650717792601440402835539414079381366983377088647344908927933831836630631716063770084998915292051693964443\\
    8216604752993305131640476262193218429800355849013973874783131117458859450866951973398339996923767105000341927\\
    0507413724770689394042753420456932178035475309918120924342950701665921597454958647624661672804060577324906410\\
    3106960592478041981149176196107935299945110957723431336120677588412270759804345145059752748778709354455807077\\
    9241685555801937414597749997315604921853640003484594638554684536433643103329082619009333334649948767697626169\\
    1094635355506745782415795477460215411399415022994437332621409310160411053748595382115305046008108563822771081\\
    3382476839175523681365921918734790740696845069888529800004638030062885348241728785175879607593973395075845978\\
    4234774772366311635678695415076974783158950140630984205944183414336313625078028992257538500247159110450478649\\
    3937790942800950628069894460928966023944090259262570867822758901705734330806434139816963277115558046690263455\\
    2394106350184502016182708127976536511085403755411758034605649316738731882039175530408172088326820821040155976\\
    9467335290388051799091203896110028012755391819293502929248840065496929139468562539714242050865537889275795815\\
    1849031709338316290703872472908609494240473701258316449135623974788101482033630070383358354071196386368974151\\
    8970669057502757227145384708144581137176442467521066865241315253252579988044799519311008558464202702056021509\\
    7803482910338543688828717693350232828437921351715260667352403016918811667569409747720909788854715680752552665\\
    0538596961848614010335479904730313210576668287916245281810162558355678941331466841926685570626424755471830179\\
    0967577242918654594420533212084944034710966020041352544337109472967109993081402208747936906368049309771970662\\
    1773173347120177738358873251159562057481454783173614764827244740235927138509987028702550955829763212659949861\\
    2934752066687488563453593694489524041716576775809226779568367984391133457584731322839659223973708077630238387\\
    9990698156467476396874615916474145967377191263214202097687842725235322928720708165104332976694711893346560080\\
    8860846459880306161274930198695951400703527757678633404528340353503417185815556251901397769044681170639539568\\
    5321268672051456721839560393540715379215317375574153005294523678342889239934758712144099942897663746508526368\\
    0024846655644856229691483984616169499410902535808168105754656121031628205616580964826962857891178918616473710\\
    9991900645449234816779145548615271655600431744768111597542274437837412507779223204608011360392543785480373479\\
    7293230327143459408511259713352268236959978903973149421399552993093323660380094954728088574486501034774479165\\
    4017997541530493676085544139453290088560911190291944303246006008903764818276499682528967311843691320096554645\\
    0393298487146127286210053050485952416002242605260047785309053791882588359924965933940798504901121815362158439\\
    0671138844902880206900914757115064615303583485153503677157831499119207578663040227739777568577494509906041296\\
    2910787175634443436178669336698435495228248530670642367771443436577619056378786471639255405948052296161225484\\
    1145545404305964166877715198539994124134368329966978464651536267068972848597760591950426573027594677142774596\\
    8406340748017015876787125768091001995281728911396767169880212700712809183575062566763709689034197355925580779\\
    7453278357346060376370825950541366930424118181961445165264136295329163972595381645515816175008336812774806061\\
    1392101215336685597714702816766848365784844089007386913062501572382442138966354848910613760126529901371200586\\
    6417400600273635195404805858467737134056484282148445517368708357347572257707700465979769191927090771393699823\\
    029479322892418710377180143250409036802278161018339108936551179\\
\end{longtable}}
\begin{proposition}\label{Prop_ex}
Let $q=p^h\equiv 1 \pmod 3$, $p\nmid R$, and consider $B\in \mathbb{F}_q$ satisfying $B^2-3B+9=0$.  Then $L_{1,B}$ has no roots in $\mu_{q^2+q+1}$. Furthermore,  the polynomial $x^{s+k(q^{k-1}+\cdots+q+1)}L(x^{q-1})$ permutes $\mathbb{F}_{q^3}$
for any $k,s\geq 0$ such that  $(s-(3kq+3j),q^2+q+1)=1$ and $(s+k(q^2+q+1),q-1)=1$.  
\end{proposition}
\proof
Let $s$ be a noncube in $\mathbb{F}_q$ and $k=s^{(q-1)/3}$. Then $\{1,\xi,\xi^2\}$ with $\xi^3=s$ is a basis of $\mathbb{F}_{q^3}$ over $\mathbb{F}_q$. We note that $k^2+k+1=0$, and so $B=-3k$ or $-3k^2$. An element $x=x_0+x_1 \xi+x_2 \xi^2$ with $x_0, x_1,x_2\in \mathbb{F}_q$ belongs to $\mu_{q^2+q+1}$ if and only if $s^2 x_2^3 - 3 s x_0 x_1 x_2 + s x_1^3 + x_0^3=0$. On the other hand, $L_{1,-3k}(x)=0$ if and only if
$\alpha_2 \xi^2+\alpha_1\xi+\alpha_0=0$ where
$$
\begin{array}{lll}
\alpha_0&=& 3 k s^2 x_2^3 - 3 k x_0^3 + s^4 x_2^6 + 3 s^3 x_0 x_1 x_2^4 - 7 s^3 x_1^3 x_2^3 -
        7 s^2 x_0^3 x_2^3 + 9 s^2 x_0^2 x_1^2 x_2^2 + 3 s^2 x_0 x_1^4 x_2 + s^2 x_1^6\\
        && +
        4 s^2 x_2^3 + 3 s x_0^4 x_1 x_2 - 7 s x_0^3 x_1^3 + 6 s x_0 x_1 x_2 - 2 s x_1^3 +
        x_0^6 + x_0^3 + 1,\\
\alpha_1&=&    3 k s^3 x_0 x_2^5 - 6 k s^3 x_1^2 x_2^4 + 3 k s^2 x_0^2 x_1 x_2^3 +
        3 k s^2 x_0 x_1^3 x_2^2 + 3 k s^2 x_1^5 x_2 - 6 k s x_0^4 x_2^2\\
        &&+
        3 k s x_0^3 x_1^2 x_2 - 6 k s x_0^2 x_1^4 + 6 k s x_0 x_2^2 - 3 k s x_1^2 x_2 +
        3 k x_0^5 x_1 + 6 k x_0^2 x_1 + 3 s^3 x_0 x_2^5 - 6 s^3 x_1^2 x_2^4\\
        &&+
        3 s^2 x_0^2 x_1 x_2^3 + 3 s^2 x_0 x_1^3 x_2^2 + 3 s^2 x_1^5 x_2 - 6 s x_0^4 x_2^2
        + 3 s x_0^3 x_1^2 x_2 - 6 s x_0^2 x_1^4 - 3 s x_0 x_2^2 - 3 s x_1^2 x_2\\ &&+
        3 x_0^5 x_1 + 6 x_0^2 x_1,\\
\alpha_2&=&    -3 k s^3 x_1 x_2^5 + 6 k s^2 x_0^2 x_2^4 - 3 k s^2 x_0 x_1^2 x_2^3 +
        6 k s^2 x_1^4 x_2^2 - 3 k s x_0^3 x_1 x_2^2 - 3 k s x_0^2 x_1^3 x_2\\
        && -
        3 k s x_0 x_1^5 - 6 k s x_1 x_2^2 - 3 k x_0^5 x_2 + 6 k x_0^4 x_1^2 -
        6 k x_0^2 x_2 + 3 k x_0 x_1^2 - 9 x_0^2 x_2.\\
\end{array}
$$
By using that $\alpha_0=\alpha_1=\alpha_2=0$ and $s^2 x_2^3 - 3 s x_0 x_1 x_2 + s x_1^3 + x_0^3=0$, and by eliminating $x_1,x_2,x_3$, we obtain that $s=0$, which is  impossible. A similar argument holds for $B=-3k^2.$ The computations are done in MAGMA.
\endproof



\section{Permutation polynomials from bijections between $\mu_{q+1}$ and $\mathbb{F}_q \cup \{ \infty\}$}\label{Sec5}

We define the image of $H(x)=\sum_{i =0}^{n}\alpha_ix^i\in \mathbb{F}_{q^2}[x]$ under the Frobenius map $\sigma$ of $\mathbb{F}_{q}$ by $H^{\sigma}(x)=\sum_{i =0}^{n}\alpha_i^q x^i$. Our generalization of Theorem~\ref{Zieve2013Theorem1.2} is the following.

\begin{theorem}\label{mainTheorem2}
Let $H,L,M \in \mathbb{F}_{q^2}[x]$ be such that $H$ is monic of degree $n$, and $L$, $M$ are $(\beta,d)$-self-associated.
Suppose that
\begin{enumerate}[{\normalfont(i)}]
\item   $L/M$ is a bijection from $\mu_{q+1}$ into $\mathbb{F}_{q}\cup \{\infty\}$,
\item the algebraic curve $\prod _{\xi \in \mathbb{F}_q}(H(x)-\xi H(y))$ has no affine $\mathbb{F}_{q}$-rational points off the line $x=y$, and
\item $H^{\sigma}-H$  has no roots in $\mathbb{F}_q$.
\end{enumerate}
For any integer $k \geq 0$,  the polynomial
$x^{dn+k(q+1)}(H\bullet L/M)(x^{q-1})$
permutes $\mathbb{F}_{q^2}$ if and only if $(dn+k(q+1),q-1)=1$.
\end{theorem}

\proof
By Theorem~\ref{ppcriterion}, the polynomial
$x^{dn+k(q+1)}(H\bullet L/M)(x^{q-1})$ permutes $\mathbb{F}_{q^2}$ if and only if $(dn+k(q+1),q-1)=1$ and
$x^{dn+k(q+1)}\left(H \bullet L/M \right)^{q-1}$
permutes $\mu_{q+1}$. Suppose that $H(x)=\sum_{i =0}^{n}\alpha_ix^i$ with $\alpha_n=1$.
We observe that the polynomial
$$H \bullet L/M=\sum_{i =0}^{n}\alpha_i L^iM^{n-i}$$
has no roots in $\mu_{q+1}$.
In fact, say that  $r \in \mu_{q+1}$ is a root of $H \bullet L/M$.  If $M(r)=0$ then $\alpha_n=1$ implies that $L(r)=0$, which is not possible since $L$ and $M$ have no common roots in $\mu_{q+1}$. If $M(r) \neq 0$ then $L(r)/M(r)$ is a root of $H$ in $\mathbb{F}_{q}$ that satisfies $0=H(L(r)/M(r))=H(L(r)/M(r))^q=H^\sigma(L(r)/M(r))$, which gives another contradiction by (iii).
In order to show that $x^{dn+k(q+1)}\left(H \bullet L/M \right)^{q-1}$ permutes $\mu_{q+1}$, take $x \in \mu_{q+1}$. Using that $M^q=\beta x^{-d}M$ and  $L^q=\beta x^{-d}L$,
we can rewrite the latter polynomial as
\begin{eqnarray*}
x^{dn+k(q+1)}\dfrac{(H \bullet L/M)^{q}}{H \bullet L/M}&=&x^{dn}\frac{\sum_{i =0}^{n}\alpha_i^q L^{qi}M^{q(n-i)}}{\sum_{i =0}^{n}\alpha_i L^iM^{n-i}}\\
&=&x^{dn}\frac{\sum_{i =0}^{n}\alpha_i^q \beta^{n}x^{-dn}L^iM^{n-i}}{\sum_{i =0}^{n}\alpha_i L^iM^{n-i}}\\
&=&\beta^{n}\frac{\sum_{i =0}^{n}\alpha_i^q L^iM^{n-i}}{\sum_{i =0}^{n}\alpha_i L^iM^{n-i}}\\
&=&\beta^{n}\frac{\sum_{i =0}^{n}\alpha_i^q (L/M)^{i}}{\sum_{i =0}^{n}\alpha_i (L/M)^{i}}\\
&=&\beta^{n}\frac{(H (L/M))^q}{H (L/M)}
\end{eqnarray*}
where the last step follows from (i).
We aim now to show that there is no pair $(x,y)\in \mu_{q+1}^2$ such that $x\neq y$ and
\begin{equation*}
\left(H\left(\frac{L(x)}{M(x)}\right)\right)^{q-1}=\left(H\left(\frac{L(y)}{M(y)}\right)\right)^{q-1}.
\end{equation*}
Let $z=L(x)/M(x)$ and $w=L(y)/M(y)$. If $z$ and $w$ are both in $\mathbb F_q$, then
$$(H(z))^{q-1}=(H(w))^{q-1}  \iff  \exists \, \xi \in \mathbb{F}_q: \  H(z)=\xi H(w),$$
which is not possible by (ii). By (i), it remains to consider $z=\infty$. If $w=\infty$ then  $x=y$. If $w\in \mathbb{F}_q$ then $H^{\sigma}(w) =
(H(w))^q=H(w)$,
a contradiction by (iii).
\endproof

We now prove Theorem~\ref{Zieve2013Theorem1.2} using Theorem~\ref{mainTheorem2}.

\medskip

\textit{Proof of Theorem~\ref{Zieve2013Theorem1.2}.}
 Consider $L(x)=\delta x-\beta\delta^q$ and $M(x)=x-\beta$ where  $\beta\in \mu_{q+1}$ and $\delta \in \mathbb{F}_{q^2}\setminus \mathbb{F}_q$. Clearly, $L$ and $M$ are  $(-\beta^{-1},1)$-self-associated. By setting  $-\beta = \epsilon$ and using Lemma~\ref{bijinfty}, we see that $L/M=(\delta \epsilon x+ (\delta\epsilon)^q)/(\epsilon x+\epsilon^q)$   is a bijection from $\mu_{q+1}$ to $\mathbb{F}_q \cup \{\infty\}$.  Let $H(x)=x^n-\delta$. We have that $\prod_{\xi\in \mathbb{F}_q} (H(x)-\xi H(y))=\prod_{\xi\in \mathbb{F}_q} (x^n-\xi y^n -\delta +\xi\delta)$ has no affine $\mathbb{F}_{q}^2$-rational points off the line $x=y$ if and only if $(n,q-1)=1$.  In fact,  this is true when $\xi \neq 1$ since $\delta\notin \mathbb{F}_q$. If $\xi=1$ then $x^n-y^n=0$ has $\mathbb{F}_{q}^2$-rational points off the line $x=y$ if and only if $(n,q-1)>1$. Furthermore, the  polynomial $\varphi(x)=(-\delta)^q-\delta$ is nonzero since $\delta \notin \mathbb{F}_q$. We conclude that
 \begin{eqnarray*}
 x^{n+k(q+1)}M(x^{q-1})^nH\left(\frac{L(x^{q-1})}{M(x^{q-1})}\right) &=& x^{n+k(q+1)}M(x^{q-1})^n \left(\frac{L(x^{q-1})^n}{M(x^{q-1})^n}-\delta\right)\\
 &=& x^{n+k(q+1)}\left(L(x^{q-1})^n-\delta M(x^{q-1})^n\right)\\
 &=& x^{n+k(q+1)}\left((\delta x^{q-1}-\beta\delta^q)^n -\delta(x^{q-1}-\beta)^n\right)
 \end{eqnarray*}
permutes $\mathbb{F}_{q^2}$ if and only if $(n,q-1)=1$ and $(n+k(q+1),q-1)=1$. The latter condition is equivalent to $(n(n+2k),q-1)=1$.
\hfill$\square$

\medskip

We can obtain more families of permutation polynomials over $\mathbb{F}_{q^2}$ by applying Theorem~\ref{mainTheorem2} with other polynomials $H$.

\begin{corollary}\label{search}
Let $f,g\in \mathbb{F}_{q}[x]$ be monic polynomials such that  $f$ has degree $n-1$ with no roots in $\mathbb{F}_q$ and $g$ permutes $\mathbb F_q$. Let $\gamma\in \mathbb{F}_{q^2}^*$ be such that $\gamma^q=-\gamma$. Then the polynomial $H=gf-\gamma f$ satisfies the following conditions:
\begin{enumerate}[{\normalfont(i)}]
\item $\prod _{\xi \in \mathbb{F}_q}(H(x)-\xi H(y))$ has no $\mathbb{F}_{q}$-rational points off the line $x=y$ and
\item $H^\sigma -H$ has no roots in $\mathbb{F}_q$.
\end{enumerate}
\end{corollary}
\proof
The polynomial  $H^\sigma(T) -H(T)=(g(T)f(T)+\gamma f(T))-(g(T)f(T)-\gamma f(T))=2\gamma f(T)$ has no roots in $\mathbb{F}_q$, by hypothesis.
Also, $H=(g-\gamma)f$. If $\xi=0$ then $H(x)-\xi H(y)=0$ gives that $H(x)=0$ and it has no roots in $\mathbb{F}_q$. If $\xi\neq 0$ then
$$H(x)-\xi H(y)=0 \iff g(x)f(x)-\gamma f(x)=\xi(g(y)yf(y)-\gamma f(y)).$$
Raising both sides to the power $q$ yields that $g(x)f(x)+\gamma f(x)=\xi(g(y)f(y)+\gamma f(y))$, which is equivalent to  $g(x)f(x)  =\xi g(y)f(y)$ and $f(x)=\xi f(y)$. Since $f(x)$ and $f(y)$ cannot be zero for $x,y \in \mathbb{F}_q$, we obtain that $g(x)=g(y)$. This implies that $x=y$ since $g$ permutes $\mathbb{F}_q$.
\endproof

Note that we cannot find self-associated polynomials of degree two over $\mathbb{F}_q$ satisfying the conditions in Corollary \ref{search}. In fact, such a polynomial would have only one root in $\mu_{q+1}$ as if it has one root in $\mu_{q+1}$, then  the other one also belongs to $\mu_{q+1}$.

\begin{ex}\label{Ex1}
 Let $q=p^{2n+1}$ where $p\equiv7,17,23 \text{ or } 33 \pmod{40}$. Consider $\alpha \in \mu_{q+1}$ such that $\alpha^2-3\alpha+1=0$, i.e.  $\alpha=\frac{3\pm \sqrt{5}}{2}$.
It is easily seen that  $\sqrt{\alpha}\notin \mu_{q+1}$ since $5$ is not a square in $\mathbb{F}_q$ ($p \equiv \pm2 \pmod{5}$). Also, $\alpha$ is a square in $\mathbb{F}_q$  such that  $\sqrt{\alpha}\notin \mu_{q+1}$ since $2$ is a square in $\mathbb{F}_q$ ($p\equiv \pm1\pmod 8$).
The polynomials $L(x)=(x-\alpha)(\alpha x^2-1)$  and $M(x)=(\alpha x-1)(x^2-\alpha)$ are $(\alpha^2,3)$-self-associated, and clearly $L/M$ is a bijection from $\mu_{q+1}$ to $\mathbb{F}_q \cup \{\infty\}$. The curve associated with $L/M$ splits as $(xy-y+1)(xy+x+1)$. The components do not have points in $\mu_{q+1}^2$.
\end{ex}

\section{Acknowledgements}
The authors thank Michael Zieve for many valuable conversations. 
 
The first author was partially supported by the Italian Ministero dell'Istruzione, dell'Universit\`a
e della Ricerca (MIUR) and  the Gruppo Nazionale per le Strutture Algebriche, Geometriche e le loro Applicazioni (GNSAGA-INdAM).

The second author received support for this project provided by a PSC-CUNY award, \#60235-00 48, jointly funded by The Professional Staff Congress and The City University of New York.

Part of this work was carried out when the authors were visiting IMPA and UFRJ in Rio de Janeiro, Brazil.  We thank the warm hospitality we received from both institutes.

\bibliographystyle{abbrv}

\begin{thebibliography}{10}

\bibitem{Bartoli2018}
D.~Bartoli.
\newblock On a conjecture about a class of permutation trinomials.
\newblock Submitted.


\bibitem{BG2018}
D.~Bartoli and M.~Giulietti.
\newblock Permutation polynomials, fractional polynomials, and algebraic curves.
\newblock Finite Fields Appl. {\bf 51}, 1--16, 2018.

\bibitem{BQ2017}
D.~Bartoli and L.~Quoos.
\newblock Permutation polynomials of the type $x^r g(x^{s})$ over ${\mathbb {F}}_{q^{2n}}$.
\newblock Des. Codes Cryptogr., DOI: 10.1007/s10623-017-0415-8.

\bibitem{Hou2015} X.~Hou.
\newblock Permutation polynomials over finite fields -- a survey of recent advances.
\newblock Finite Fields Appl. {\bf 32}, 82--119, 2015.

\bibitem{L2007}
Y.~Laigle-Chapuy.
\newblock Permutation polynomials and applications to coding theory.
\newblock Finite Fields Appl. {\bf 13}, 58--70, 2007.

\bibitem{LM1984}
R.~Lidl and W. B.~M\"{u}ller.
\newblock Permutation polynomials in RSA-cryptosystems.
\newblock Advances in Cryptology, Springer, Boston, MA, 1984.

\bibitem{LN1997}
R.~Lidl and H.~Niederreiter.
\newblock Finite Fields, second ed., Encyclopedia Math. Appl., vol. 20, Cambridge University Press, Cambridge,
1997.574

\bibitem{M1993}
G.L.~Mullen.
\newblock Permutation polynomials over finite fields.
\newblock In: Proc. Conf. Finite Fields and Their Applications, in: Lect. Notes
Pure Appl. Math. {\bf 141}, Marcel Dekker, 131--151, 1993.

\bibitem{MP2013}
G.L.~Mullen and D.~Panario.
\newblock Handbook of Finite Fields.
\newblock Chapman and Hall/CRC, 2013.

\bibitem{SH1998}
J.~Schwenk and  K.~Huber.
\newblock Public key encryption and digital signatures based on permutation polynomials.
\newblock Electronics Letters {\bf 34}(8),  759--760, 1998.

\bibitem{TZ2000}
T.J.~Tucker and M.E.~Zieve.
\newblock Permutation polynomials, curves without points, and Latin squares.
\newblock preprint, 2000.

\bibitem{Zieve2013}
M.E.~Zieve.
\newblock Permutation polynomials on $\mathbb{F}_q$ induced from bijective R\'edei functions on subgroups of the multiplicative group of $\mathbb{F}_q$.
\newblock {\em arXiv:1310.0776v2, 7 Oct 2013.}


\end{thebibliography}

\end{document}